\documentclass[11pt]{amsart}

\usepackage[a4paper]{geometry}
\usepackage{latexsym, amsthm, amsfonts, amsmath,amssymb,graphicx,lmodern,mathrsfs}

\usepackage[latin1]{inputenc} %%% Pour les accents
\usepackage[T1]{fontenc} %%% pour la césure des mots accentués
\usepackage[english]{babel} %%% toujours en dernier

\newtheorem{theo}{Theorem}%[section]
\newtheorem{lem}[theo]{Lemma}
\newtheorem{propo}[theo]{Proposition}

\theoremstyle{definition}

\theoremstyle{remark}
\newtheorem{rem}[theo]{Remark}

\def\R{\mathbb{R}}
\def\Z{\mathbb{Z}}
\def\C{\mathbb{C}}
\def\N{\mathbb{N}}

\def\n{\eta}

\def\a{\alpha}
\def\e{\varepsilon}
\def\d{\delta}

\def\t{\theta}

\def\n'{\nu}
\def\d{\delta}

\def\k{\kappa}

\def\L{\Lambda}
\def\T{\Theta}

\def\sq {\sigma_{q}} 
\def\sp {\sigma_{p}} 
\def\sqe {\sigma_{q_{0}^{\e}}} 
\def\dqe {\delta_{q_{0}^{\e}}}

\begin{document}
\sloppy
\title{Isomonodromic deformation of $q$-difference equations and confluence}
\author{Thomas Dreyfus}
\address{Universit\'e Claude Bernard Lyon 1, Institut Camille Jordan, 
43 boulevard du 11 novembre 1918,
69622 Villeurbanne, France.}
\email{dreyfus@math.univ-lyon1.fr}
\thanks{Work supported by the labex CIMI. This project has received funding from the European Research Council (ERC) under the European Union's Horizon 2020 research and innovation programme under the Grant Agreement No 648132.}



\date{\today}

\begin{abstract}
We study isomonodromic deformation of Fuchsian linear $q$-difference systems. Furthermore, we are looking of the behaviour of the Birkhoff connexion matrix when $q$ goes to $1$. We use our results to study the convergence of the Birkhoff connexion matrix that appears in the definition of the $q$-analogue of the sixth Painlevé equation. 
\end{abstract} 

\maketitle
\tableofcontents
%\dedicatory{u}

\section*{Introduction}
The study of isomonodromic deformation and Painlevé equations in the framework $q$-difference equations have obtained many contributions. See \cite{Bor,GNP,KK,KMN,NRG,PNG,RGH,Sak,S06}. In \cite{JS}, Jimbo and Sakai have introduced a $q$-analogue of Painlevé sixth equation. Analogously to what happens in the differential case, the equation appears after considering an isomonodromic  deformation of a $q$-difference equation. More precisely, let $q\in \C^{*}$ with $|q|>1$, and let us define the $q$-difference operator $\sq$, 
$$\sq \big(f(z)\big):=f(qz).$$ 
A $q$-difference equation may be seen as a discretization of a differential equation, since $\frac{\sq^{-1}-\mathrm{Id}}{q^{-1}-1}$ converges to the derivation $z\frac{\mathrm{d}}{\mathrm{d}z}$ when $q$ goes to $1$. This leads Jimbo and Sakai to consider a $q$-deformation of the differential equation that is relevant for the sixth Painlevé equation:

\begin{equation}\label{eq0}
\sq^{-1} Y(z,t)=(z-1)(z-t)(\mathrm{I}+(q^{-1}-1)z\mathcal{A}(z,t))Y(z,t)=A(z,t)Y(z,t),
\end{equation}
where $\mathcal{A}(z,t):=\frac{\mathcal{A}_{0}}{z}+\frac{\mathcal{A}_{1}}{z-1}+\frac{\mathcal{A}_{t}}{z-t}$, $\mathcal{A}_{0},\mathcal{A}_{1},\mathcal{A}_{t}$ are $2\times 2$ complex matrices, $t$ denotes a parameter that belongs to $U$, an open connected subset of $\C^{*}$ stable under $\sq$, and $\mathrm{I}$ denotes the identity matrix. Then, under convenient assumptions, Jimbo and Sakai construct invertible matrix solutions, $Y_{0},Y_{\infty}$, at $z=0$ and $z=\infty$ of (\ref{eq0}) and consider the Birkhoff connexion matrix $P(z,t):=Y_{\infty}^{-1}(z,t)Y_{0}(z,t)$, which play a role that is analogous to that of the monodromy matrices for differential equations. Then, the authors of \cite{JS} find a necessary and sufficient condition so that the matrix $P$ is pseudo constant, that is for all $t\in U$, $P(z,t)=P(z,qt)$. The condition they found is the existence of an invertible matrix $z\mapsto B(z,t)$ having coefficients in $\C(z)$, such that 
$$Y_{0}(z,qt)=B(z,t)Y_{0}(z,t), \hbox{ and } Y_{\infty}(z,qt)=B(z,t)Y_{\infty}(z,t).$$
Moreover, the matrix $B$ satisfies
$$
A(z,qt)B(z,t)=B(qz,t)A(z,t),
$$
which leads Jimbo and Sakai to the definition of the $q$-analogue of the sixth Painlevé equation.\\ \par 
Many objects that appear in the Galois theory of $q$-difference equations may be seen as $q$-analogues of the corresponding objects that appear in the Galois theory of differential equations. See \cite{Ad29,Be,Ca,D4,DE,DR,DSK,DV02,GR,MZ,R92,Ro11,RS14,RSZ,RZ,S00,Trj,vdPR,Z02}. The convergence of the operator $\frac{\sq-\mathrm{Id}}{q-1}$ to  $z\frac{\mathrm{d}}{\mathrm{d}z}$ when $q$ goes to $1$ leads to the study of the confluence\footnote{Throughout the paper, we will use the word ``confluence'' to describe the $q$-degeneracy when $q\rightarrow 1$.} of these objects. Confluence of solutions of Fuchsian linear $q$-difference systems has been considered in \cite{S00}.
Let us summarize the work of Sauloy. Considering a convenient family of $q$-difference systems of the form $$\sq Y(z,q)=(\mathrm{I}+(q-1)zA(z,q))Y(z,q),$$  how do behave the solutions $Y_{0},Y_{\infty}$ at $z=0$ and $z=\infty$, when $q$ goes to $1$? Contrary to \cite{JS}, to construct the solutions $Y_{0}$ and $Y_{\infty}$, the author of \cite{S00} uses only functions that are meromorphic on $\C^{*}$, rather than multivalued functions. Assume that $A(z,q)$ converges to $\widetilde{A}(z)$ in a convenient way. Under reasonable assumptions, Sauloy proves that the univalued solutions $Y_{0},Y_{\infty}$ converge when $q$ goes to $1$ in a convenient way, to multivalued solutions of the linear differential system ${\frac{\mathrm{d}}{\mathrm{d}z}\widetilde{Y}(z)=\widetilde{A}(z)\widetilde{Y}(z)}$. Then, he obtains that the Birkhoff connexion matrix $P(z,q)=Y_{\infty}^{-1}(z,t)Y_{0}(z,t)$ converges to a locally constant matrix $\widetilde{P}(z)$ when $q$ goes to $1$. In \cite{S00}, it is shown that the monodromy matrices at the intermediates singularities (those different from $0$ and $\infty$) of the linear differential system ${\frac{\mathrm{d}}{\mathrm{d}z}\widetilde{Y}(z)=\widetilde{A}(z)\widetilde{Y}(z)}$ can be expressed with the values of~$\widetilde{P}(z)$. \\ \par 
The main goal of this paper is the following. Can we use the reasoning of Sauloy to study the behaviour  when $q$ goes to $1$, of the Birkhoff connexion matrix that is involved in the definition of the $q$-analogue of the sixth Painlevé equation? 
\\ \par 
The paper is organized as follows. In $\S \ref{sec1}$, we make a reminder of the local study of Fuchsian linear $q$-difference equations. We introduce the meromorphic functions we will use to solve Fuchsian linear $q$-difference systems and we define the Birkhoff connexion matrix. In $\S \ref{sec2}$, we adapt the work of Jimbo and Sakai to a different situation and with using meromorphic solutions on $\C^{*}$ rather than multivalued functions. In particular, we study in general isomonodromic deformation of Fuchsian linear $q$-difference systems. In $\S \ref{sec3}$ we use the work of Sauloy to state a result of convergence of the  Birkhoff connexion matrix when $q$ goes to $1$. See Theorem \ref{theo1}. Finally in $\S \ref{sec4}$, we apply this result to (\ref{eq0}), the equation that leads the authors of \cite{JS} to the $q$-analogue of the sixth Painlevé equation. Unfortunately, (\ref{eq0}) does not behave very well when $q$ goes to $1$. In particular, we prove that the Birkhoff connexion matrix of (\ref{eq0}) equals to the product of a Birkhoff connexion matrix of another system, that converges when $q$ goes to $1$, and matrices that do not behave well when $q$ goes to $1$, but that we might express explicitly. See Theorem \ref{theo2}. Note that further applications of this theorem could be to study analytic properties of the $q$-analogue of the sixth Painlevé equation.

\pagebreak[3]
\section{Fuchsian linear $q$-difference systems}\label{sec1}
We are going to remind in this section how to solve Fuchsian linear $q$-difference systems using meromorphic functions on $\C^{*}$. Let us fix $q$ a complex number with $|q|>1$. For $R$ a ring and $\nu\in \N^{*}$, let $\mathrm{GL}_{\nu}(R)$ be the field of invertible $\nu\times \nu$ matrices in coefficients in $R$. Let us fix $\log$, a determination of the logarithm over $\widetilde{\C}$, the Riemann surface of the logarithm. For $a\in \C$, we are going to write $q^{a}$ instead of  $e^{a\log(q)}$. For $S\subset \R$, let $q^{S}:=\{q^{\ell},\ell\in S\}$. Let us consider the linear $q$-difference system of rank  $\nu\in \N^{*}$:
\begin{equation}\label{eq1}
\sq Y(z)=R(z)A(z)Y(z), \hbox{ with } R\in \C(z)\setminus \{0\}, A\in \mathrm{GL}_{\nu}\left(\C(z)\right).
\end{equation}
Let us assume that $A$ has no poles at $0$ and $\infty$. Let us also assume that $A_{0}:=A(0)$, (resp. $A_{\infty}:=A(\infty)$) is invertible.\par 
To solve (\ref{eq1}), we are going to introduce three meromorphic functions on~$\C^{*}$,~$$\T_{q}(z):=\displaystyle \sum_{n \in \Z} q^{\frac{-n(n+1)}{2}}z^{n}=\displaystyle \prod_{n=0}^{\infty}\left(1-q^{-n-1}\right)\left(1+q^{-n-1}z\right)\left(1+q^{-n}z^{-1}\right),$$ $\L_{q,a}(z):=\frac{\T_{q}(z)}{\T_{q}(z/a)}$ with~$a\in \C^{*}$, and $l_{q}(z):=\T_{q}(z)^{-1}z\frac{\mathrm{d}}{\mathrm{d}z}\T_{q}(z)$.
They satisfy the~$q$-difference equations: 
\begin{itemize}
\item~$\sq \T_{q}(z)=z\T_{q}(z)$.
\item~$\sq \L_{q,a}(z)=a\L_{q,a}(z)$.
\item~$\sq l_{q}=l_{q} +1$.
\end{itemize}
The theta function $\T_{q}(z)$ is analytic on $\C^{*}$ and has zeroes of order $1$ in the $q$-spiral $-q^{\Z}$. The function $\L_{q,a}$ has poles of order $1$ in the $q$-spiral $-aq^{\Z}$, and zeroes of order $1$ in the $q$-spiral $-q^{\Z}$.\par 
Let~$B$ be an invertible matrix with complex coefficients and consider now the decomposition in Jordan normal form~$B=P(DU)P^{-1}$, where~$D:=\mathrm{Diag}(d_{i})$ is diagonal,~$U$ is a unipotent upper triangular matrix with~$DU=UD$, and~$P$ is an invertible matrix with complex coefficients. Following \cite{S00}, we construct the matrix:
$$
\L_{q,B}:=P\left(\mathrm{Diag}\left(\L_{q,d_{i}}\right)e^{\log(U)l_{q}}\right)P^{-1}\in \mathrm{GL}_{m}\Big(\C\left( l_{q},\left(\L_{q,b}\right)_{b \in \C^{*}}\right)\Big)$$
that satisfies:
$$
\sq \L_{q,B}=B\L_{q,B}=\L_{q,B}B.$$
Note that the matrix $\L_{q,B}$ depends implicitly upon the choice of the change of basis matrix~$P$. Let~$b\in \C^{*}$ and consider the corresponding matrix~$(b)\in \mathrm{GL}_{1}(\C)$. By construction, we have~${\L_{q,b}=\L_{q,(b)}}$.\par 
 Let $\mu_{0}$ be the valuation of $R$ (resp. $\mu_{\infty}$ be the degree of $R$) et let $r_{0}$ (resp. $r_{\infty}$) be the value of $z^{-\mu_{0}}R$ at $z=0$ (resp. $z^{-\mu_{\infty}}R$ at $z=\infty$). Assume that the distinct eigenvalues of $A_{0}$ (resp. $A_{\infty}$) are distinct modulo $q^{\Z}$. Let $\C\{z\}$ be the ring of germs of analytic functions at $z=0$. As we can see in \cite{S00},~$\S 1$, there exist two invertible matrices with entries that are meromorphic on $\C^{*}$, solution of (\ref{eq1}), of the form 
$$Y_{0}(z)=\hat{H}_{0}(z)\L_{q,r_{0}A_{0}}\T_{q}(z)^{\mu_{0}}, $$
$$Y_{\infty}(z)=\hat{H}_{\infty}(z)\L_{q,r_{\infty}A_{\infty}}\T_{q}(z)^{\mu_{\infty}}, $$
where $\hat{H}_{0}(z),\hat{H}_{\infty}\left(z^{-1}\right) \in \mathrm{GL}_{\nu}(\C\{z\})$, and $\hat{H}_{0}\left(0\right)=\hat{H}_{\infty}\left(\infty\right)=\mathrm{I}$.\par 
Let $a_{1},\dots,a_{m}$ be the poles of $A$, (resp. $\t_{1},\dots,\t_{\nu}$ be the eigenvalues of $A_{0}$). From \cite{S00}, $\S 1$, it follows that 
\begin{itemize}
\item $Y_{0}$ has its poles contained in the $q$-spirals $-q^{\Z},-r_{0}\t_{1}q^{\Z},\dots,-r_{0}\t_{\nu}q^{\Z}$, and $a_{1}q^{\Z_{>0}},\dots,a_{m}q^{\Z_{>0}}$.%, and $\det (Y_{0})$ has all its zeroes contained  in the $q$-spirals $-q^{\Z}$, ${b_{i}q^{\Z_{>0}}}$, for $i\in \{1,\dots, n\}$.
\item $Y_{\infty}^{-1}$ has its poles contained in the $q$-spirals $-q^{\Z}$, and $a_{1}q^{\Z_{\leq 0}},\dots,a_{m}q^{\Z_{\leq 0}}$.%,  and $\det (Y_{\infty}^{-1})$ has all its zeroes, all of order $1$,  in the $q$-spirals $-p_{\mu}\k_{1}q^{\Z},\dots,-p_{\mu}\k_{\nu}q^{\Z}$, ${b_{i}q^{\Z_{<0}}}$, for $i\in \{1,\dots, n\}$.
\end{itemize}

The Birkhoff connection matrix is defined by 
$$P(z):=Y_{\infty}^{-1}(z)Y_{0}(z).$$
It satisfies $\sq P=P$, i.e., the entries of $P$  belong to the field of meromorphic functions over the torus $\C^{*}\setminus q^{\Z}$. This field can be identified with the field of elliptic functions. Therefore, the entries of $P$ may be written in terms of theta functions. Note that it plays an analogue role to that of the monodromy matrices for differential equations. The entries of $P$ are meromorphic on $\C^{*}$ with poles contained in the $q$-spirals, $-q^{\Z},-r_{0}\t_{1}q^{\Z},\dots,-r_{0}\t_{\nu}q^{\Z}$, ${a_{1}q^{\Z},\dots,a_{m}q^{\Z}}$.
%, and $\det(P)$ has its zeroes, all of order $1$, contained in the $q$-spirals  $-\k_{1}q^{\Z}, \dots,-\k_{\nu}q^{\Z}$, ${b_{i}q^{\Z}}$, for $i\in \{1,\dots, n\}$.

\pagebreak[3]
\section{Connection preserving deformation}\label{sec2}
In this section, we are going to prove that under convenient assumptions, the work of Jimbo and Sakai, see $\S \ref{sec4}$, stay valid in a different situation and with using the meromorphic solutions of $\S \ref{sec1}$ rather than multivalued functions.\par 
 Following \cite{JS}, we consider the linear $q$-difference system,
\begin{equation}\label{eq2}\begin{array}{lll}
\sq Y(z,t)=R(z,t)A(z,t)Y(z,t),
\end{array}\end{equation}
with $z\mapsto R(z,t)\in \C(z)\setminus\{0\}$ and $z\mapsto A(z,t)\in \mathrm{GL}_{\nu}\left(\C(z)\right)$. Here $t$ denotes a complex parameter belonging to $U\subset \C^{*}$, an open connected set that is stable under $\sq$. 
Let us assume that:
\begin{itemize}
\item For all $t\in U$, the matrix $A$ has no poles at $0$ and $\infty$.  Let us also assume that $A_{0}:=A(0,t)$, (resp. $A_{\infty}:=A(\infty,t)$) is invertible, does not depend upon $t$, with distinct eigenvalues that are distinct modulo $q^{\Z}$.
\item Each poles of $z\mapsto A(z,t)$ is proportional to $t$ or does not depend upon $t$.  
 \item The valuation and the degree of $R$ are independent of $t$. Let $\mu_{0}$ be the valuation of $R$ (resp. $\mu_{\infty}$ be the degree of $R$). Assume that $r_{0}(t)$ (resp. $r_{\infty}(t)$), the value of $z^{-\mu_{0}}R(z,t)$ at $z=0$ (resp. $z^{-\mu_{\infty}}R(z,t)$ at $z=\infty$) satisfies for all $t\in U$, $r_{0}(qt)\in r_{0}(t)q^{\Z}$ (resp. $r_{\infty}(qt)\in r_{\infty}(t)q^{\Z}$).
\end{itemize}
 Let 
$$Y_{0}(z,t)=\hat{H}_{0}(z,t)\L_{q,r_{0}(t)A_{0}}\T_{q}(z)^{\mu_{0}}, $$
$$Y_{\infty}(z,t)=\hat{H}_{\infty}(z,t)\L_{q,r_{\infty}(t)A_{\infty}}\T_{q}(z)^{\mu_{\infty}}, $$
be the solutions of (\ref{eq2}) defined in $\S \ref{sec1}$. Let $a_{1}(t),\dots,a_{m}(t)$
be the poles of $A(z,t)$. Without loss of generalities, we may assume the existence of integer $m_{1}$, such that $a_{1}(t),\dots,a_{m_{1}}(t)$ are proportional to $t$ and $a_{m_{1}+1}(t),\dots,a_{m}(t)$ do not depend upon $t$. Now, we are going to determine a necessary and sufficient condition so that, for all values of the parameter $t\in U$, the Birkhoff connection matrix ${P(z,t):=Y_{\infty}^{-1}(z,t)Y_{0}(z,t)}$ is pseudo constant, that is for all $t\in U$,
$$P(z,t)=P(z,qt).$$
The next proposition is analogous to \cite{JS}, Proposition 2. 

\pagebreak[3]
\begin{propo}\label{propo1}
The Birkhoff connection matrix $P$ is pseudo constant if and only if there exists $z\mapsto B(z,t)\in \mathrm{GL}_{\nu}(\C(z))$ such that 
$$Y_{0}(z,qt)=B(z,t)Y_{0}(z,t), \hbox{ and } Y_{\infty}(z,qt)=B(z,t)Y_{\infty}(z,t).$$
Moreover, in this case, $B(z,t)$ has its poles in $\C^{*}$ that are equal to $a_{1}(t),\dots,a_{m_{1}}(t)$, and we have
\begin{equation}\label{eq3}
A(z,qt)B(z,t)=B(qz,t)A(z,t).
\end{equation}
\end{propo}

\begin{proof}
The Birkhoff connection matrix $P$ is pseudo constant if and only if 
$$B(z,t):=Y_{\infty}(z,qt)Y_{\infty}(z,t)^{-1}=Y_{0}(z,qt)Y_{0}(z,t)^{-1}.$$
 Using the description of the poles of $Y_{0}$ and $Y_{\infty}^{-1}$, their construction, see $\S \ref{sec1}$, and the assumptions on $r_{0}(t),r_{\infty}(t),A_{0},A_{\infty}$, we find that $B(z,t)$ is meromorphic on $\C^{*}$ and has its poles in $\C^{*}$ that are equal to $a_{1}(t),\dots,a_{m_{1}}(t)$. The assumption we have made on $r_{0}(t)A_{0}$ implies that, ${z\mapsto \L_{q,r_{0}(qt)A_{0}}\L_{q,r_{0}(t)A_{0}}^{-1}\in \mathrm{GL}_{\nu}(\C(z))}$. Similarly, ${z\mapsto \L_{q,r_{\infty}(qt)A_{\infty}}\L_{q,r_{\infty}(t)A_{\infty}}^{-1}\in \mathrm{GL}_{\nu}(\C(z))}$. We remind that for all $t\in U$, ${\hat{H}_{0}\left(0,t\right)=\hat{H}_{\infty}\left(\infty,t\right)=\mathrm{I}}$. Hence, we find that, for all $t\in U$, the matrix $B$ is meromorphic at ${z=0}$ (resp. the matrix $B$ is meromorphic at $z=\infty$), since 
$$
\begin{array}{lll}
B(z,t)&=&\hat{H}_{0}(z,qt)\L_{q,r_{0}(qt)A_{0}}\L_{q,r_{0}(t)A_{0}}^{-1}\hat{H}_{0}(z,t)^{-1}\\
&=&\hat{H}_{\infty}(z,qt)\L_{q,r_{\infty}(qt)A_{\infty}}\L_{q,r_{\infty}(t)A_{\infty}}^{-1}\hat{H}_{\infty}(z,t)^{-1}.
\end{array}
$$ Therefore, ${z\mapsto B(z,t)\in \mathrm{GL}_{\nu}(\C(z))}$. The relation (\ref{eq3}) follows from $$Y_{0}(qz,t)=A(z,t)Y_{0}(z,t), \hbox{ and } Y_{0}(z,qt)=B(z,t)Y_{0}(z,t).$$
\end{proof}

\pagebreak[3]
\section{Behaviour of the Birkhoff connection matrix when $q$ goes to $1$}\label{sec3}
In this section, we see $q$ as a parameter. We are going to make the matrix $A(z,t)$ of $\S\ref{sec2}$ depends upon $q$ and see what is the behaviour of the Birkhoff connection matrix when $q$ goes to $1$. First of all, let us fix $q_{0}\in \C^{*}$ with $|q_{0}|>1$. Inspiriting from the ideas of Sauloy in \cite{S00}, we are going to put $q:=q_{0}^{\e}$ with $\e\in \R_{>0}$ and make $\e$ goes to~$0^{+}$. For $R$ a ring, let $\mathrm{M}_{\nu}\left(R\right)$ be the ring of $\nu\times \nu$ matrices with coefficients in $R$. Formally $\d_{q}:=\frac{\sq-\mathrm{Id}}{(q-1)}$,  converges to the derivation $z\frac{\mathrm{d}}{\mathrm{d}z}$ when $q$ goes to $1$. Then, it is natural to consider 
\begin{equation}\label{eq4}\dqe Y(z,t,\e)=A(z,t,\e)Y(z,t,\e), \hbox{ with }z\mapsto A(z,t,\e)\in \mathrm{M}_{\nu}\left(\C(z)\right).
\end{equation}
Here, $t$ belongs to $ U\subset \C^{*}$, that is an open connected set which is for all $\e>0$ stable under $\sqe$. Let us assume that:
\begin{itemize}
\item For all $t\in U$, for all $\e >0$ sufficiently close to $0$, the matrix $A$ has no poles at $0$ and $\infty$.  Let us also assume that $A_{0}(\e):=A(0,t,\e)$, (resp. $A_{\infty}(\e):=A(\infty,t,\e)$) does not depend upon $t$ and its Jordan normal form depends continuously upon~$\e$. Assume the existence of complex invertible matrices $Q_{0}(\e),Q_{\infty}(\e)$ with entries that are continuous in $\e$, such that for all ${\e\in \R_{>0}}$ close to $0$, the matrices $Q_{0}(\e)A_{0}(\e)Q_{0}(\e)^{-1}$ and $Q_{\infty}(\e)A_{\infty}(\e)Q_{\infty}(\e)^{-1}$ are in Jordan normal form. 
\item  Each poles of $z\mapsto A(z,t,\e)$ is proportional to $t$ or does not depend upon~$t$. Let $a_{1}(t,\e),\dots,a_{m}(t,\e)$
be the poles of $A(z,t,\e)$. Let us also assume that $a_{1}(t,\e),\dots,a_{m}(t,\e)$ converge to $t\mapsto\widetilde{a_{1}}(t),\dots,t\mapsto\widetilde{a_{m}}(t)\in \C^{*}$ when ${\e\to 0^{+}}$.
\item For every $t\in U$, for every compact subset of $\mathbb{P}_{1}(\C)\setminus \left\{ \widetilde{a_{1}}(t),\dots,\widetilde{a_{m}}(t)\right\}$, we have the uniform convergence of $A(z,t,\e)$ to $z\mapsto \widetilde{A}(z,t)\in \mathrm{M}_{\nu}\left(\C(z)\right)$ when $\e\to 0^{+}$. 
\item For all $t\in U$, the differential system $z\frac{\mathrm{d}}{\mathrm{d}z} \widetilde{Y}(z,t)=\widetilde{A}(z,t)\widetilde{Y}(z,t)$ is Fuchsian on $\mathbb{P}_{1}(\C)$ and has exponents at~$0$ and $\infty$ which are non resonant, which means that the eigenvalues of $\widetilde{A}_{0}:=\widetilde{A}(0,t)$ (resp. $\widetilde{A}_{\infty}:=\widetilde{A}(\infty,t)$) do not differ by a non zero integer\footnote{ Note that combined with the previous assumptions, this implies that for every $\e>0$ sufficiently small, the distinct eigenvalues of $\mathrm{I}+(q_{0}^{\e}-1)A_{0}(\e)$ (resp. $\mathrm{I}+(q_{0}^{\e}-1)A_{\infty}(\e)$) are distinct modulo $q_{0}^{\e\Z}$.}.

\item   For all $t\in U$, the complex numbers $$-1,\widetilde{a_{1}}(t),\dots,\widetilde{a_{m}}(t)$$
 are distinct modulo ${q_{0}^{\R}}$.
\end{itemize}
Note that after choosing another complex number $q_{0}$ with $|q_{0}|>1$, and taking another open connected set $U$, we always may reduce to the case where the last assumption is satisfied. Remark also that for every $\e$ sufficiently close to $0$, (\ref{eq4}) satisfies the same assumptions as (\ref{eq2}) in $\S \ref{sec2}$.
Let 
$$Y_{0}(z,t,\e)=\hat{H}_{0}(z,t,\e)\L_{q_{0}^{\e},\mathrm{I}+(q_{0}^{\e}-1)A_{0}(\e)}, $$
$$Y_{\infty}(z,t,\e)=\hat{H}_{\infty}(z,t,\e)\L_{q_{0}^{\e},\mathrm{I}+(q_{0}^{\e}-1)A_{\infty}(\e)}, $$
be the solutions of (\ref{eq4}) defined in $\S \ref{sec1}$. We would like to study the behaviour of the two solutions when $\e\to 0^{+}$.  
%Let $\widetilde{\t_{1}},\dots,\widetilde{\t_{\nu}},\widetilde{\k_{1}},\dots,\widetilde{\k_{\nu}}$ be the limits of $\t_{1}(\e),\t_{2}(\e),\k_{1}(\e),\k_{2}(\e)\in \C^{*}$ when $\e\to 0^{+}$. 
%Let $$\widetilde{\Omega}_{t}:=\C^{*}\setminus \left\{-t\widetilde{\t_{1}}q_{0}^{\R},-t\widetilde{\t_{2}}q_{0}^{\R},t\widetilde{a_{1}}q_{0}^{\R},t\widetilde{a_{2}}q_{0}^{\R} ,\widetilde{a_{3}}q_{0}^{\R},\widetilde{a_{4}}q_{0}^{\R},-q_{0}^{\R}\right\}.$$
For $A\in \mathrm{M}_{\nu}(\C)$, we are going to write $z^{A}$ instead of $e^{A\log(z)}$.  In this situation, the work of Sauloy, see \cite{S00}, $\S 3$, may be applied and we find:
\pagebreak[3]
\begin{propo}\label{propo2}
\begin{trivlist}
\item (1)
For all $t\in U$, there exists $z\mapsto\widetilde{H}_{0}(z,t)\in \mathrm{GL}_{\nu}(\C\{z\})$, with ${\widetilde{H}_{0}(0,t)=\mathrm{I}}$ such that $\hat{H}_{0}(z,t,\e)\L_{q_{0}^{\e},\mathrm{I}+(q_{0}^{\e}-1)A_{0}(\e)}$ converges uniformly to $\widetilde{H}_{0}(z,t)z^{\widetilde{A}_{0}}$ in every compact subset of $$\widetilde{\Omega}_{0,t}:=\C^{*}\setminus \left\{-q_{0}^{\R},\widetilde{a_{1}}q_{0}^{\R_{>0}},\dots,\widetilde{a_{m}}q_{0}^{\R_{>0}}\right\},$$ when $\e\to 0^{+}$. 
\item (2)
For all $t\in U$, there exists $z\mapsto\widetilde{H}_{\infty}(z,t)\in \mathrm{GL}_{\nu}(\C\{z^{-1}\})$, with ${\widetilde{H}_{\infty}(\infty,t)=\mathrm{I}}$,  such that $\hat{H}_{\infty}(z,t,\e)\L_{q_{0}^{\e},\mathrm{I}+(q_{0}^{\e}-1)A_{\infty}(\e)}$ converges uniformly to $\widetilde{H}_{\infty}(z,t)z^{\widetilde{A_{\infty}}}$ in every compact subset of $$\widetilde{\Omega}_{\infty,t}:=\C^{*}\setminus \left\{-q_{0}^{\R},\widetilde{a_{1}}q_{0}^{\R_{<0}},\dots,\widetilde{a_{m}}q_{0}^{\R_{<0}}\right\},$$ when $\e\to 0^{+}$. 
\end{trivlist}
\end{propo}

Let $$\widetilde{Y}_{0}(z,t):=\widetilde{H}_{0}(z,t)z^{\widetilde{A_{0}}},$$

$$\widetilde{Y}_{\infty}(z,t):=\widetilde{H}_{\infty}(z,t)z^{\widetilde{A_{\infty}}},$$

which are the solutions obtained via the Frobenius algorithm of 
$$z\frac{\mathrm{d}}{\mathrm{d}z} \widetilde{Y}(z,t)=\widetilde{A}(z,t)\widetilde{Y}(z,t).
$$
It follows that $P(z,t,\e)$ converges uniformly to $\widetilde{P}(z,t):=\widetilde{Y}_{\infty}(z,t)^{-1}\widetilde{Y}_{0}(z,t)$ in every compact subset of $$\widetilde{\Omega}_{t}:=\C^{*}\setminus \left\{-q_{0}^{\R},\widetilde{a_{1}}q_{0}^{\R},\dots,\widetilde{a_{m}}q_{0}^{\R}\right\},$$ when $\e\to 0^{+}$. Moreover, the matrix $\widetilde{P}(z,t)$ is locally constant, since it is meromorphic  and satisfies  $\frac{\mathrm{d}}{\mathrm{d}z} \widetilde{P}(z,t)=0$.

\begin{figure}[ht]
\begin{center}
\includegraphics[width=0.8\linewidth]{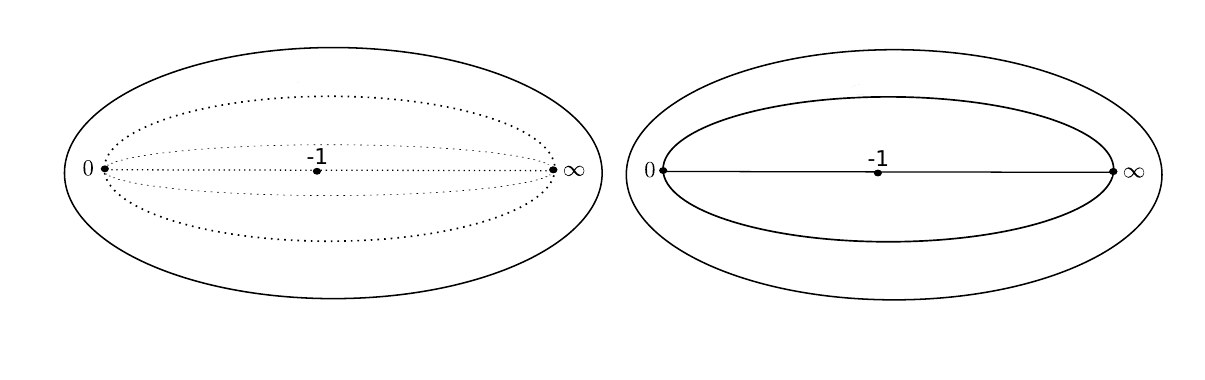}
\caption{Domain of definition of $z\mapsto P(z,t,\e)$ (left), and $z\mapsto \widetilde{P}(z,t)$ (right).}
\end{center}
\end{figure}

Let~$\widetilde{a}_{0}=\widetilde{a}_{m+1}:=-1$. We remind that  for all $t\in U$, the complex numbers $-1,\widetilde{a_{1}}(t),\dots,\widetilde{a_{m}}(t)$
 are distinct modulo ${q_{0}^{\R}}$. Let $\phi$ be the bijection of $\{0,\dots,m+1\}$, such that $\phi(0)=0$, $\phi(m+1)=m+1$, and for all $t\in U$, a positive circle around $z=0$ starting at $z=-1$ intersect the spirals $\widetilde{a}_{\phi(1)}(t)q_{0}^{\R},\dots,\widetilde{a}_{\phi(m)}(t)q_{0}^{\R}$ in this order. For $j\in \{0,\dots,m\}$, let $\widetilde{U}_{j}$ be the open connected subset of $\C^{*}$ with border $\widetilde{a}_{\phi(j)}(t)q_{0}^{\R}\cup\widetilde{a}_{\phi(j+1)}(t)q_{0}^{\R}$.
Note that the connected component of the domain of definition of~$\widetilde{P}$ are the~$\widetilde{U}_{j}$ with $j\in \{0,\dots,m\}$. Let~${z\mapsto \widetilde{P}_{j}(z,t)\in \mathrm{GL}_{\nu}(\C)}$ be the value of~$\widetilde{P}$ in~$\widetilde{U}_{j}$. 

\pagebreak[3]
\begin{theo}[\cite{S00}, $\S 4$]\label{theo1}
Let $j\in \{1,\dots,m\}$. The monodromy matrix of the linear differential equation~${z\frac{\mathrm{d}}{\mathrm{d}z} \widetilde{Y}(z,t)=\widetilde{A}(z,t)\widetilde{Y}(z,t)}$ in the basis~$\widetilde{Y}_{0}(z,t)$ around the singularity~$\widetilde{a}_{\phi(j)}(t)$ is~$\left(\widetilde{P}_{j}\right)^{-1}\widetilde{P}_{j-1}$.
\end{theo}

\pagebreak[3]
\section{Application to the $q$-analogue of the sixth Painlevé equation}\label{sec4}
We begin this section, by reminding the work of Jimbo and Sakai, see \cite{JS}, in the framework of meromorphic functions on $\C^{*}$, rather than multivalued functions. Let us consider the $2\times 2$ linear differential system that is relevant for the sixth Painlevé equation:
$$\frac{\mathrm{d}}{\mathrm{d}z}Y(z,t)=\mathcal{A}(z,t)Y(z,t), \hbox{ with } \mathcal{A}(z,t):=\frac{\mathcal{A}_{0}}{z}+\frac{\mathcal{A}_{1}}{z-1}+\frac{\mathcal{A}_{t}}{z-t}.$$
Here, $t$ denotes a parameter that belongs to $U$, an open connected subset of $\C^{*}$, and $\mathcal{A}_{0},\mathcal{A}_{1},\mathcal{A}_{t}$ are $2\times 2$ complex matrices.  Let us fix $q$ a complex number with $|q|>1$, assume that $U$ is stable under $\sq$. Let $p=1/q$ and $\sp:=\sq^{-1}$. Following \cite{JS}, we consider the linear $q$-difference system,
\begin{equation}\label{eq6}\begin{array}{lll}
\sp Y(z,t)&=&(z-1)(z-t)(\mathrm{I}+(p-1)z\mathcal{A}(z,t))Y(z,t)\\
 &=&\left(A_{0}(t)+A_{1}(t)z+\left(\mathrm{I}+(p-1)A_{2}\right)z^{2}\right)Y(z,t)=A(z,t)Y(z,t). 
\end{array}\end{equation}
Moreover, let us assume that:
\begin{itemize}
\item $A_{0}(t)$ is diagonalizable and has eigenvalues, counted with multiplicity, $t\t_{1}$ and $t\t_{2}$. Therefore, there exists $B_{0}\in \mathrm{GL}_{2}(\C)$, such that $A_{0}(t)=tB_{0}$.
\item For all $t\in U$, $A_{1}(t)$ is a complex $2\times 2$ matrix.
\item $A_{2}:=\begin{pmatrix}
\k_{1}&0\\
0&\k_{2}
\end{pmatrix}$. 
\item For all $t\in U$, $t\t_{1},t\t_{2}$, (resp. $1+(p-1)\k_{1},1+(p-1)\k_{2}$) are equal or are distinct modulo $q^{\Z}$\footnote{Note that this assumption is not present in the work of Jimbo and Sakai. The reason of this extra assumption is due to the fact that we will use the confluence results of \cite{S00} in $\S \ref{sec3}$ that force us to consider $q$-difference systems that are not resonant.}.
\item There exist $a_{1},\dots,a_{4}\in \C^{*}$ such that $$\det(A(z,t))=(1+(p-1)\k_{1})(1+(p-1)\k_{2})(z-ta_{1})(z-ta_{2})(z-a_{3})(z-a_{4}).$$
%\item For all $t\in U$, the complex numbers $$-1,ta_{1},ta_{2},a_{3},a_{4}$$
% are distinct modulo $q^{\Z}$.
\end{itemize}
Note that for $p$ sufficiently close to $1$, (\ref{eq6}) satisfies the same assumptions as (\ref{eq2}) in $\S \ref{sec2}$ and we may apply Proposition~\ref{propo1}. Therefore, the Birkhoff connection matrix is pseudo constant if and only if we have the existence of a convenient matrix $z\mapsto B(z,t)\in \mathrm{GL}_{2}(\C(z))$ that satisfies properties described in Proposition~\ref{propo1}. Moreover, $B(z,t)$ has only simple poles that are equal to $ta_{1},ta_{2}$ if $a_{1}\neq a_{2}$, (resp. $B(z,t)$ has only a double pole that is equal to $ta_{1}$ if $a_{1}= a_{2}$)  and we have
$$
A(z,qt)B(z,t)=B(qz,t)A(z,t).
$$
The latter system leads Jimbo and Sakai to the definition of the $q$-analogue of the sixth Painlevé equation. \\ \par 
As in $\S \ref{sec3}$ we would like to make (\ref{eq6}) depends upon $q$ and see what is the behaviour of the Birkhoff connection matrix when $q$ goes to $1$. Unfortunately, (\ref{eq6}) does not degenerate very well when $q$ goes to $1$, due to the factor $(z-1)(z-t)$. Therefore, we are going to solve the factor $(z-1)(z-t)$ separately.  
As in $\S \ref{sec3}$, let us fix a complex number $q_{0}$ with $|q_{0}|>1$. Put $q:=q_{0}^{\e}$ and let us make $\e\in \R_{>0}$ goes to $0^{+}$ in the equation 

\begin{equation}\label{eq7}
\sigma_{q_{0}^{-\e}} Y(z,t,\e)=\left(\mathrm{I}+(q_{0}^{-\e}-1)z\mathcal{A}(z,t)\right)Y(z,t,\e).
\end{equation}
Assume that ${\frac{\mathrm{d}}{\mathrm{d}z} \widetilde{Y}(z,t)=\mathcal{A}(z,t)\widetilde{Y}(z,t)}$ has exponents at~$0$ and $\infty$ which are non resonant. Moreover, let us assume that for all $t\in U$, 
$-1,1,t$, are distinct modulo $q_{0}^{\R}$, and $U$ is stable under $\sigma_{q_{0}^{\e}}$ for every $\e\in \R_{>0}$. Note that for $\e$ sufficiently close to $0$, (\ref{eq7}) satisfies the assumptions of (\ref{eq4}) in $\S \ref{sec3}$. Let 
$$Y_{0}(z,t,\e)=\hat{H}_{0}(z,t,\e)\L_{q_{0}^{\e},B_{0}^{-1}(\e)}, $$
$$Y_{\infty}(z,t,\e)=\hat{H}_{\infty}(z,t,\e)\L_{q_{0}^{\e},\left(\mathrm{I}+(q_{0}^{-\e}-1)A_{2}\right)^{-1}}, $$
with $z\mapsto \hat{H}_{0}(z,t,\e)\in \mathrm{GL}_{2}(\C\{z\})$, $z\mapsto \hat{H}_{\infty}(z,t,\e)\in \mathrm{GL}_{2}(\C\{z^{-1}\})$ and ${\hat{H}_{0}(0,t,\e)=\hat{H}_{\infty}(\infty,t,\e)=\mathrm{I}}$ be the solutions of (\ref{eq7}) defined in $\S \ref{sec1}$. Let $\widetilde{Y}_{0}(z,t),\widetilde{Y}_{\infty}(z,t)$, be the solutions of ${\frac{\mathrm{d}}{\mathrm{d}z}Y(z,t)=\mathcal{A}(z,t)Y(z,t)}$ defined in $\S \ref{sec3}$.
Let ${P(z,t,\e)=Y_{\infty}(z,t,\e)^{-1}Y_{0}(z,t,\e)}$, be the Birkhoff connection matrix of (\ref{eq7}). Due to Proposition \ref{propo2}, for all $t\in U$, we have the uniform convergence of $P(z,t,\e)$ to the locally constant matrix $\widetilde{P}(z,t)=\widetilde{Y}_{\infty}(z,t)^{-1}\widetilde{Y}_{0}(z,t)$ in every compact subset of $$\widetilde{\Omega}_{t}:=\C^{*}\setminus \left\{-q_{0}^{\R},q_{0}^{\R},tq_{0}^{\R}\right\}$$ when $\e\to 0^{+}$.  Moreover, Theorem \ref{theo1} may be applied, and the values of $\widetilde{P}(z,t)$ in the connected component of its domain of definition allow us to recover the monodromy matrices at the singularities $z=1$ and $z=t$ of $\frac{\mathrm{d}}{\mathrm{d}z}Y(z,t)=\mathcal{A}(z,t)Y(z,t)$. 
\pagebreak[3]
\begin{rem}
Note that $P$ is also a Birkhoff connexion matrix of 
\begin{equation}\label{eq8}\begin{array}{lll}
\sigma_{q_{0}^{-\e}} Y(z,t,\e)&=&(z-1)(z-t)(\mathrm{I}+(q_{0}^{-\e}-1)z\mathcal{A}(z,t))Y(z,t,\e)\\
 &=&\left(tB_{0}(\e)+A_{1}(t,\e)z+\left(\mathrm{I}+(q_{0}^{-\e}-1)A_{2}\right)z^{2}\right)Y(z,t,\e). 
\end{array}\end{equation} 
Unfortunately, $P$ is different to $Q$, the Birkhoff connexion matrix of (\ref{eq8}), that is defined in~$\S \ref{sec1}$. To simplify the notations, from now, we are going to write $q$ instead of $q_{0}^{\e}$ and $p=1/q$ instead of $q_{0}^{-\e}$. The goal of what follows is to prove the following theorem:
\end{rem}
 
\pagebreak[3]
\begin{theo}\label{theo2}
The Birkhoff connection matrix ${Q(z,t,\e)}$ of (\ref{eq8}) that is defined in $\S \ref{sec1}$ equals to
\begin{equation}\label{eq9}
 \frac{\T_{q}(qz)^{2}\prod_{n=0}^{\infty}\left(1-q^{-n-1}\right)^{2}}{\T_{q}\left(-qz\right)\T_{q}\left(\frac{-qz}{t}\right)}\left(\mathrm{I}+(p-1)A_{2}\right)^{2}P(z,t,\e)\left(\L_{q,B_{0}^{-1}(\e)}\right)^{-1}\frac{\L_{q,t^{-1}B_{0}^{-1}(\e)}}{q^{3}}.
\end{equation}
Furthermore, the Birkhoff connection matrix $Q(z,t,\e)$ is pseudo constant if and only if $P(z,qt,\e)=-P(z,t,\e)t^{2}B_{0}(\e)$.
\end{theo}
 Before proving the theorem, let us sate and prove a lemma.
 
\pagebreak[3]
\begin{lem}\label{lem1}
Let $\a\in\C^{*}\setminus \left\{-q^{\Z}\right\}$. Let $y_{\a,0},y_{\a,\infty}$, be the two solutions of $\sp y(z)=(z-\a)y(z)$ described in $\S \ref{sec1}$. Then, we have  
$$y_{\a,0}:=\frac{\T_{q}(z)}{\displaystyle \prod_{n=0}^{\infty}\left(1-q^{-n-1}\frac{qz}{\a}\right)\T_{q}(-\a z)},y_{\a,\infty}:=\frac{\displaystyle \prod_{n=0}^{\infty}\left(1-q^{-n}\frac{\a}{qz}\right)}{\T_{q}(qz)},$$
and
$$y_{\a,\infty}^{-1}y_{\a,0}=\displaystyle \prod_{n=0}^{\infty}\left(1-q^{-n-1}\right)\frac{\T_{q}(z)\T_{q}(qz)}{\T_{q}\left(\frac{-qz}{\a}\right)\T_{q}(-\a z)}.$$
\end{lem}

\begin{proof}[Proof of Lemma \ref{lem1}]
First of all, note that $\sp y(z)=(z-\a)y(z)$ is equivalent to ${\sq y=\frac{1}{qz-\a}y}$.  Therefore, there exist $\hat{h}_{0}(z) \in \C\{z\}$, and $\hat{h}_{\infty}\left(z\right)\in \C\{z^{-1}\}$ with ${\hat{h}_{0}\left(0\right)=\hat{h}_{\infty}\left(\infty\right)=1}$ such that 
$$y_{\a,0}:=\frac{\hat{h}_{0}(z)\T_{q}(z)}{\T_{q}(-\a z)},\hbox{ and }y_{\a,\infty}:=\frac{\hat{h}_{\infty}(z)}{\T_{q}(qz)}.$$
The function $\hat{h}_{0}$ is solution of $\sq \left(\hat{h}_{0}\right)=\dfrac{1}{1-\frac{qz}{\a}}h_{0}$. Since $h_{0}(0)=1$, we have 
$$h_{0}=\displaystyle \prod_{n=0}^{\infty}\left(1-q^{-n-1}\frac{qz}{\a}\right)^{-1}. $$
Similarly, $\sq \left(\hat{h}_{\infty}\right)=\dfrac{1}{1-\frac{\a}{qz}}h_{\infty}$ and  
$$h_{\infty}=\displaystyle \prod_{n=0}^{\infty}\left(1-q^{-n}\frac{\a}{qz}\right). $$
We now use the Jacobi triple product formula
$$\T_{q}\left(\frac{-qz}{\a}\right)=\displaystyle \prod_{n=0}^{\infty}\left(1-q^{-n-1}\right)\left(1-q^{-n-1}\frac{qz}{\a}\right)\left(1-q^{-n}\frac{\a}{qz}\right),$$
to deduce that 
$$h_{\infty}^{-1}h_{0}\displaystyle \prod_{n=0}^{\infty}\left(1-q^{-n-1}\right)^{-1}=\T_{q}\left(-\frac{qz}{\a}\right)^{-1}.$$
This completes the proof.
\end{proof}

\begin{proof}[Proof of Theorem \ref{theo2}]
Let $W_{0}(z,t,\e),W_{\infty}(z,t,\e)$ be the solutions of (\ref{eq8}) that are defined in~$\S \ref{sec1}$. Since (\ref{eq8}) is equivalent to 
$$\sq Y(z,t,\e)=\left(tB_{0}(\e)+A_{1}(t,\e)qz+\left(\mathrm{I}+(p-1)A_{2}\right)q^{2}z^{2}\right)^{-1}Y(z,t,\e),$$
we find 
$$W_{0}(z,t,\e)=\hat{F}_{0}(z,t,\e)\L_{q,t^{-1}B_{0}^{-1}(\e)}, $$
$$W_{\infty}(z,t,\e)=\hat{F}_{\infty}(z,t,\e)\L_{q,\left(\mathrm{I}+(p-1)A_{2}\right)^{-1}q^{-2}}\T_{q}(z)^{-2}, $$
with $z\mapsto \hat{F}_{0}(z,t,\e)\in \mathrm{GL}_{2}(\C\{z\})$, $z\mapsto \hat{F}_{\infty}(z,t,\e)\in \mathrm{GL}_{2}(\C\{z^{-1}\})$ and ${\hat{F}_{0}(0,t,\e)=\hat{F}_{\infty}(\infty,t,\e)=\mathrm{I}}$.
Using Lemma~\ref{lem1}, we find  
$$W_{0}(z,t,\e)=y_{1,0}(z,\e)y_{t,0}(z,t,\e)Y_{0}(z,t,\e)\left(\L_{q,B_{0}^{-1}(\e)}\right)^{-1}\frac{\T_{q}(-z)\T_{q}(-tz)\L_{q,t^{-1}B_{0}^{-1}(\e)}}{\T_{q}(z)^{2}} $$
and 
$$
\begin{array}{lll}
W_{\infty}(z,t,\e)&=&y_{1,\infty}(z,\e)y_{t,\infty}(z,t,\e)Y_{\infty}(z,t,\e)\left(\L_{q,\left(\mathrm{I}+(p-1)A_{2}\right)^{-1}}\right)^{-1}\frac{\T_{q}(qz)^{2}\L_{q,\left(q^{2}\left(\mathrm{I}+(p-1)A_{2}\right)\right)^{-1}}}{\T_{q}(z)^{2}}\\
&=&y_{1,\infty}(z,\e)y_{t,\infty}(z,t,\e)Y_{\infty}(z,t,\e)q^{3}\left(\mathrm{I}+(p-1)A_{2}\right)^{-2}.
\end{array}
 $$
We now apply Lemma~\ref{lem1}, to prove that the Birkhoff connection matrix ${Q(z,t,\e)=W_{\infty}^{-1}(z,t,\e)W_{0}(z,t,\e)}$ equals to
$$\displaystyle \prod_{n=0}^{\infty}\left(1-q^{-n-1}\right)^{2}\frac{\T_{q}(qz)^{2}}{\T_{q}\left(-qz\right)\T_{q}\left(\frac{-qz}{t}\right)}\left(\mathrm{I}+(p-1)A_{2}\right)^{2}P(z,t,\e)\left(\L_{q,B_{0}^{-1}(\e)}\right)^{-1}\frac{\L_{q,t^{-1}B_{0}^{-1}(\e)}}{q^{3}}.$$
This proves (\ref{eq9}).
To conclude the proof of the theorem, we just have to show that the Birkhoff connection matrix $Q(z,t,\e)$ is pseudo constant if and only if ${P(z,qt,\e)=-P(z,t,\e)t^{2}B_{0}(\e)}$. We use (\ref{eq9}) to deduce that the Birkhoff connection matrix $Q(z,t,\e)$ is pseudo constant if and only if
$$\frac{P(z,qt,\e)\L_{q,(qt)^{-1}B_{0}^{-1}(\e)}}{\T_{q}\left(\frac{-z}{t}\right)}=\frac{P(z,t,\e)\L_{q,t^{-1}B_{0}^{-1}(\e)}}{\T_{q}\left(\frac{-qz}{t}\right)}.$$
This is equivalent $P(z,qt,\e)=-P(z,t,\e)t^{2}B_{0}(\e)$, which proves the result.
\end{proof}

\pagebreak[3]
\begin{rem}
Instead of the Birkhoff connexion matrix, following \cite{S15}, one could prefer to study the behaviour of the matrix $\hat{F}_{\infty}(z,t,\e)^{-1}\hat{F}_{0}(z,t,\e)$. 
Applying Proposition \ref{propo2}, we find that for all $t\in U$, there exist $z\mapsto\widetilde{H}_{0}(z,t)\in \mathrm{GL}_{2}(\C\{z\})$, $z\mapsto\widetilde{H}_{\infty}(z,t)\in \mathrm{GL}_{2}(\C\{z^{-1}\})$, with $\widetilde{H}_{0}(0,t)=\widetilde{H}_{\infty}(\infty,t)=\mathrm{I}$ such that $\hat{H}_{\infty}(z,t,\e)^{-1}\hat{H}_{0}(z,t,\e)$ converges uniformly to $\widetilde{H}_{\infty}(z,t)^{-1}\widetilde{H}_{0}(z,t)$ in every compact subset of $\widetilde{\Omega}_{t}$ when $\e\to 0^{+}$. Moreover as we can see in the proof of Lemma \ref{lem1}, we obtain that 
$$\hat{F}_{\infty}(z,t,\e)^{-1}\hat{F}_{0}(z,t,\e)=\frac{\prod_{n=0}^{\infty}\left(1-q^{n-1}\right)^{2}}{\T_{q}\left(-qz\right)\T_{q}\left(\frac{-qz}{t}\right)}\hat{H}_{\infty}(z,t,\e)^{-1}\hat{H}_{0}(z,t,\e).$$
\end{rem}

\bibliographystyle{alpha}
\bibliography{biblio}
\end{document}